\documentclass{amsart}
\usepackage{amssymb}

\usepackage{amsfonts}
\usepackage[all,2cell]{xy} \UseAllTwocells \SilentMatrices
\usepackage{graphicx}
\usepackage{amsthm}
\usepackage{amstext}
\usepackage{amsmath}
\usepackage{amscd}
\usepackage[mathscr]{eucal}
\usepackage{url}

\newtheorem*{theorem*}{Theorem}
\newtheorem{theorem}{Theorem}[section]
\newtheorem{proposition}[theorem]{Proposition}

\newtheorem{lemma}[theorem]{Lemma}

\theoremstyle{definition}
\newtheorem{definition}[theorem]{Definition}
\newtheorem{example}[theorem]{Example}

\theoremstyle{remark}

\numberwithin{equation}{section}

\newcommand\mL{L\kern-0.08cm\char39}

\def\N{{\mathbb N}}
\def\Z{{\mathbb Z}}
\def\Q{{\mathbb Q}}
\def\R{{\mathbb R}}

\begin{document}

\begin{large}

\title[On Properties Expansive Group Actions]{On Properties Expansive Group Actions}

\author[A. Barzanouni]{Ali Barzanouni}
\address{Department of Mathematics, School of Mathematical Sciences, Hakim Sabzevari University, Sabzevar, Iran}
\email{alibarzanouni@yahoo.com}
\author[Mahin Sadat Divandar]{Mahin Sadat Divandar}
\address{Department of Mathematics, School of Mathematical Sciences, Hakim Sabzevari University, Sabzevar, Iran}
\email{md.divandar@gmail.com}

\author[E. Shah]{Ekta Shah}
\address{Department of Mathematics, Faculty of Science, The Maharaja Sayajirao University of Baroda, Vadodara, India}
\email{ekta19001@gmail.com}






\begin{abstract}
We study several properties of expansive group actions on metric spaces and obtain relation between expansivity for subgroup and group actions. Through counter examples necessity of hypothesis are justified. We also study expansivity of covering maps. Further, we define orbit expansivity for group actions on topological spaces and use it to characterize expansive action.

\end{abstract}
\maketitle

\section{Introduction}\label{S:Intro}

\noindent Let $(X,d)$ be a metric space and $h : X \longrightarrow X$ be a homeomorphism. Then there is a natural action of the additive group of integers, $(\Z, +)$, by this homeomorphism $h$ on  $X$ given by $(n,x)\rightsquigarrow h^n(x)$, $n \in \mathbb{Z}$ and $x \in X$. The classical theory of dynamical systems is more inclined to the study of this $\mathbb{Z}-$action, if the system is a discrete dynamical system. If the group $\mathbb{Z}$ is replaced by the additive group of real numbers, $(\mathbb{R}, +)$, then we study the dynamics of flows (i.e. continuous time evolution systems).  However, in last two decades there has been significant  development in the study of higher dimensional actions including actions of $\mathbb{Z}^d$ or $\mathbb{R}^d$ with $d > 1$. Some of the good recent references in this direction are \cite{mr,po,sd}.

\bigskip
\noindent  It is now more than a decade, wherein the study of topological aspect of the dynamics of actions of groups $\Z^d$ or $\R^d$, $d\geq1$, is replaced by arbitrary groups in general and finitely generated groups in particular. We wish to note here that the algebraic aspect of these actions is much older and is now a well established part of dynamical systems.

\bigskip
\noindent Recently, in 2014, Osipov and Tikhomirov \cite{os}, introduced the notion of shadowing property for finitely generated group actions. Barzanouni in \cite{ali} used the notion of shadowing property to study chain recurrent sets whereas Chung and Lee in \cite{Lee} showed that expansive action which has shadowing property are always topologically stable. Hurder in \cite{SH} discusses the dynamics of expansive actions on the unit circle $S^1$. In this paper, we introduce orbit expansive group actions and obtain a characterization of expansive group action  in terms of orbit expansive group actions.

\bigskip

\noindent  A homeomorphism  $h : X \longrightarrow X$ is said to be an \textit{expansive homeomorphism} provided there exists a real number $c>0$ such that whenever $x, y \in X$ with $x\not=y$ then there exists an integer $n$ (depending on $x, y$) satisfying  $d(h^n(x), h^n(y))>c$ \cite{AH}. Constant $c$ is called an \textit{expansive constant} for $h$. Since its inception expansiveness has been extensively studied in its own right and also its relation to other dynamical properties are studied. One of the important aspects of expansive dynamical system is the study of its various generalizations and variations in different setting. For instance, Kato defined notion of continuum--wise expansive homeomorphism \cite{HK} whereas measure expansivity was introduced and studied in \cite{pv}. Ruchi Das studied the notion of expansive homeomorphism on metric $G-$spaces. Tarun Das \emph{et al.} \cite{TKR} used the notion of expansive homeomorphism on topological space to prove the Spectral Decomposition Theorem on non--compact spaces. Achigar \emph{et al.} studied the notion of orbit expansivity on non--Hausdorff space \cite{art}. In this paper we study several properties expansive group actions from topological aspect.

\bigskip
\noindent The paper is organized in the following manner. In Section 2, we discuss preliminaries related to groups and group actions. Various properties of expansive group actions are studied in Section 3. Let $H$ be subgroup of a topological group $G$. Then $\varphi : G\times X \longrightarrow X$ is an expansive action of $G$ on $X$ if $\varphi_{|H} : H\times X \longrightarrow X$ is expansive. Through examples we show that converse is not true and also obtain condition under which converse is true. Further, in the same section we show that expansivity is preserved under the conjugacy if the space is compact.  We also study expansivity of covering map in Section 3. The notion for orbit expansivity for homeomorphisms was first defined in \cite{art}. In the last Section of paper we define concept of orbit expansive group actions and characterize expansive group actions using orbit expansive group actions. We also show that if there is an orbit expansive group action on a topological space then the space is always a $T_1-$space.

\section{Preliminaries}\label{S:preli}
\subsection{Group Theory}

\medskip
\noindent Let $G$ be a group and $H$ be subgroup of $G$. For $g \in G$, the subset $gH := \{gh : h\in H \}$ (respectively $Hg := \{hg : h\in H \}$) of $G$ is called \emph{left coset} (respectively \emph{right coset}) of $H$ in $G$. A subgroup $H$ is said to be \emph{normal subgroup} of $G$ if $gH = Hg$ for all $g \in G$. If $H$ is normal subgroup of $G$, then the set of all left cosets of $H$ in $G$ i.e. $\{gH : g \in G\}$ is a group and is known as the \emph{quotient group}. We denote it by $G/H$. The number of left cosets of $H$ in $G$ is known as the \emph{index of $H$ in $G$} and is denoted by $i_G(H)$. If $G$ is finite then it follows that $i_G(H) = o(G)/o(H)$. Here $o(G)$ denotes the order of a group $G$.

\medskip
\noindent A subset $S$ of $G$ is said to be a \emph{generating set} if for any $g \in G$, there exists a finite number of elements $s_1, s_2,\ldots, s_n \in S\cup S^{-1}$ such that $g=s_1s_2.....s_n$. Here the set $S^{-1} = \{s^{-1} : s \in S\}$. A group is said to be \emph{finitely generated} if its generating set is finite set. A generating set is known as a \emph{symmetric} if  $s^{-1} \in S$ for every $s \in S$. Observe that it is always convenient to describe group by a generating set.

\medskip
\noindent A non--empty subset $H$ of a topological group $G$ is said to be a \emph{syndetic set},  if there exists compact set $K$ such that $G=KH$. Note that this set $K$ need not be a subgroup of $G$. Further, without loss of generality we can assume that set $K$ is a symmetric set. For, if $K$ is not symmetric then we add all those $k^{-1}$ in $K$ so that it is a symmetric set. Observe that the new set is still a compact  set with $G=KH$, because  the map taking $g$ to its inverse, $g^{-1}$, is continuous map on $G$.   In the following result we obtain a necessary and sufficient condition for a subset to be a syndetic set.

\smallskip
\begin{proposition}\label{syndetic}
  Let $G$ be a topological group. Then a non--empty subset $H$ of $G$ is a syndetic set if and only if there exists a compact set $K$ such that for all $g\in G$, $Kg\cap H\neq \emptyset$.
\end{proposition}
\begin{proof}
 Suppose $H$ is a syndetic set. Therefore there is a compact symmetric set $K$ such that $G=KH$. This implies each $t \in G$ can be written as $t=kh$, for some $k \in K$ and $h \in H$. Hence $k^{-1}t = h \in H$. Also, $k^{-1}t \in Kt$. Thus, $K\:t\cap H\neq \emptyset$. But $t \in G$ is arbitrary. Therefore for each $g \in G$, $Kg \cap H\neq \emptyset$.

\medskip
\noindent Conversely, suppose there is a compact set $K$ such that for every $g\in G$, $Kg\cap H\neq \emptyset$. Without loss of generality we can assume that $K$ is a symmetric set. Let $t \in G$. Then there are $k \in K$ and $h \in H$ such that $h = kt$. This implies there is  $h \in H$ and $k^{-1} \in K$ such that $t = k^{-1}h$. Hence $t \in KH$. But $t$ in $G$ is arbitrary. Therefore, $G=KH$. Thus there exists a finite set $K$ such that $G=KH$. Therefore, $H$ is a syndetic set.
\end{proof}

\smallskip
\subsection{Group Action}
\noindent  By a group here we mean a topological group though we do not specify topology on $G$ unless and otherwise required.

\smallskip
\begin{definition}
Let $X$ be a topological space and $G$ be a (topological) group. A continuous map $\varphi: G \times X \longrightarrow X$ is said to be a \emph{(left) action of a group} if the following are satisfied:
\begin{itemize}
\item[(1)]  $\varphi(g_{1}\:g_{2},\:x)= \varphi(g_{1},\varphi(g_{2},\:x))$, for all $ g_1,\:g_2 \in G, x \in X$.
\item[(2)]  $ \varphi(e,x)= x$, for all  $ x \in X$. Here $e$ is the identity element of $G$.
\end{itemize}
\end{definition}

\medskip
\noindent Note that for $g\in G$, $\varphi(g,\cdot) : X \longrightarrow X$ is a homeomorphism. We denote this $\varphi(g,\cdot)$ also by $\varphi_g$. All our actions are always continuous actions.


\section{Properties of Expansive Group Action}

\medskip
\begin{definition}
Let $X$ be a metric space with metric $d$ and $G$ be a topological group. An action $\varphi:G\times X\rightarrow X$ is said to be \emph{expansive} if there exists a constant $c>0$ such that for all $x, y \in X$ with $x\neq y$ there exists $g \in G$ satisfying $d(\varphi_g(x), \varphi_g(y))>c$. Constant $c$ is called \emph{an expansive constant for $\varphi$}. Equivalently, $\varphi$ is expansive if for each $g \in G$, $d(\varphi_g(x), \varphi_g(y))\leq c$, then $x=y$.
\end{definition}

\medskip
\noindent Let $H$ be a subgroup of group $G$ and $\varphi:G\times X\longrightarrow X$ be an action of $G$ on $X$. Suppose $\varphi_{|H} :H\times X\longrightarrow X$ is expansive, then obviously $\varphi$ is expansive. In the following example we observe that in general converse is not true.

\medskip
\begin{example}\label{expansive}
Let $\mathbb{T}^{2}=\mathbb{R}^{2}/\mathbb{Z}^{2}$ denote the additive $2-$dimensional torus. Note that any continuous surjective endomorphism $f:\mathbb{T}^{2} \longrightarrow \mathbb{T}^{2}$  is of the form $f(x)= f_A(x)$ where $A$ is a non-singular $2\times 2$ integer matrix and $f_A$ is defined by
$$ f_A(x)= Ax \ (mod \: 1)$$
The invertible transformation $f_A$ is expansive if and only if $A$ has no eigenvalue of modulus one \cite{AH}. Take
\begin{equation}
B=\left(%
\begin{array}{cc}
  -1 & 1 \\
  0 & 1 \\
\end{array}%
\right)  \;\& \;\;
C=\left(%
\begin{array}{cc}
  -1 & 0 \\
  1 & 1 \\
\end{array}%
\right) \ and \ G=<C, B>
\end{equation}
 Let the action of $G$, $\varphi:G\times \mathbb{T}^{2} \longrightarrow \mathbb{T}^{2}$ be given by $\varphi(A, x)= Ax \ (mod \:1)$. Note that eigenvalues of the matrix $BC$ are not of unit modulus and therefore if $H$ is subgroup  generated by $BC$ (i.e. $H=<BC>$), then it follows that $\varphi|_H$ is expansive. Hence $\varphi:G\times \mathbb{T}^{2}\longrightarrow \mathbb{T}^{2}$ is expansive. But neither $\varphi|_{<B>}$ or $\varphi|_{<C>}$ is expansive.
\end{example}

\medskip
\noindent If $H$ is subgroup of finite index of a group $G$, then by Proposition \ref{syndetic} obviously $H$ is a syndetic set. Note that in Example \ref{expansive}, neither of the subgroups $<B>$ or $<C>$ are syndetic sets.  In the following result we obtain conditions under which expansivity of a group implies expansivity of subgroup.

\medskip
\begin{proposition}
  Let $X$ be a compact metric space and let $H$ be a syndetic subgroup of  $G$. Suppose $\varphi:G\times X\longrightarrow X$ is an action of $G$ on $X$. Then $\varphi$ is expansive if and only if $\varphi_{| H}$ is expansive.
\end{proposition}
\begin{proof} It is sufficient to prove that if $\varphi$ is an expansive action of $G$ on $X$, then $\varphi_{|H}$ is expansive. Since $H$ is syndetic for  $G$, there is  compact set $K\subseteq G$ such that $G=KH$. For $c>0$, there is $\delta>0$ such that
\begin{equation}\label{index}
 d(x, y)<\delta \Longrightarrow d(\varphi_{g}(x), \varphi_{g}(y))<c , \ \mbox{for } g\in K
 \end{equation}
Let $x,y \in X$ with $x\neq y$. Then using expansivity of  $\varphi$, it follows that there is $g\in G$ such that $d(\varphi_g(x), \varphi_g(y))>c$. But $g= g_ih$, for some $g_i\in K$ and some $h \in H$. Therefore using  inequality \ref{index} we obtain that $d(\varphi_h(x), \varphi_h(y))>\delta$. Hence $\varphi_{|H}$ is expansive with expansive constant $\delta$.
\end{proof}

\medskip
\noindent Consider $\R$ with usual topology. Then the additive group of real numbers is a topological group. Also, the set of rational numbers is a syndetic subgroup of $\R$ as there exists a compact subset $[0, 1]$ of $\R$ such that $\R=[0,1] \Q$. Note that $\Q$ is not a subgroup of finite index in $\R$. We now recall the definition of conjugate actions.
\begin{definition}
  Let  $\varphi : G\times X\longrightarrow X$ and $\psi : G\times Y\longrightarrow Y$ be two actions. Then $\varphi$ and $\psi$ are said to be \emph{topologically conjugate} if there exists a homeomorphism $h : X\longrightarrow Y$ such that  $h\circ \varphi_g=\psi_g\circ h$, for all $g\in G$. Homeomorphism $h$ is called as\emph{ conjugacy} between $\varphi$ and $\psi$.
\end{definition}
\medskip
\noindent Note that the condition $h\circ \varphi_g=\psi_g\circ h$ can also be written as $h(\varphi(g,x))= \psi(g,h(x))$, for all $g\in G$ and $x\in X$.
\noindent In the following Theorem we show that expansivity is preserved under conjugacy if the space is compact.

\medskip
\begin{theorem}\label{exp}
Let $(X,d)$ and $(Y, \rho)$ be two compact metric spaces and let $G$ be a group. Suppose $\varphi : G\times X\longrightarrow X$ and $\psi : G\times Y\longrightarrow Y$ are conjugate actions with conjugacy $h : X \longrightarrow Y$. If $\varphi$ is expansive then so is $\psi$.
\end{theorem}
\begin{proof}
Let $\varphi$  be an expansive action with expansive constant $c$. For $c>0$ there is $\delta>0$ such that
\begin{equation}
d(x, y)<\delta \Longrightarrow \rho(h(x), h(y))<c
\end{equation}
It is easy to see that $\psi$ is expansive on $Y$ with expansive constant $\delta$.
\end{proof}

\medskip
\noindent A metric space $X$ with a metric $d$ is said to satisfy  \emph{Property }$\textrm{P}$ if  for every $\epsilon>0$ there exists a compact subset $C$ of $X$ such that $d^{-1}([0, \epsilon))\cup C\times C= X\times X$. Here
\begin{equation*}
d^{-1}([0, \epsilon))= \{(x, y)\in X\times X: d(x, y)<\epsilon\}
\end{equation*}

\noindent Obviously every compact metric space satisfies  Property $\textrm{P}$, but converse need not be true in general. For  instance, consider $X=(0, 1)$ with usual metric of $\R$, the set of real numbers. Then $X$ is not compact. But for every $\epsilon>0$ it follows that
$$d^{-1}([0, \epsilon))\cup [\epsilon, 1-\epsilon]\times [\epsilon, 1-\epsilon]= (0, 1)\times (0, 1).$$
Hence $X$ satisfies Property $\textrm{P}$. In order to prove next theorem we need the following Lemma:
\medskip
\begin{lemma}\label{expc}
Let $G$ be a finitely generated group with generating set $S$ and $X$ be a  metric space satisfying Property $\textrm{P}$. Suppose the action $\varphi: G \times X \to X$ is an expansive action with expansive constant $c$. Then for any $\epsilon >0$, there exists a non--empty finite subset $F$ of $G$ such that whenever $d(\varphi_g(x), \varphi_g(y))<c$, for all  $g\in F$, then $d(x,y)<\epsilon$
\end{lemma}
\begin{proof}
If possible there is $\epsilon>0$ such that for any finite set $F$ of $G$, there exist $x_F, y_F\in X$ such that
\begin{equation*}
  \sup_{g\in F}d((\varphi_g(x_F), \varphi_g(y_F))< c \ \text{and} \ \  d(x_F, y_F)\geq\epsilon.
\end{equation*}
Choose a sequence of finite subsets $F_n$ of $G$ satisfying
\begin{equation*}
  F_1\subseteq F_2\subseteq \ldots, \text{ and } G=\bigcup_{n\in\mathbb{N}} F_n.
\end{equation*}
This implies for every $n\in\N$, there are $x_n, y_n$ in $X$ such that
\begin{equation}\label{5}
d(x_n, y_n)\geq\epsilon \text{ and}  \sup_{g\in F_n}d(gx_n, gy_n)<c
\end{equation}
But $X$ satisfy Property $\textrm{P}$. Therefore for $\epsilon>0$ there is  compact set
$C\subseteq X$ such that $$d^{-1}([0, \epsilon))\cup C\times C= X\times X.$$
Inequality \ref{5} implies that sequences $\{x_n\}, \{y_n\}$ must be in  $C$. Compactness of $C$ implies that there are subsequences of $\{x_n\}$ and $\{y_n\}$ which are convergent. If $x_{n_k}\to x$ and $y_{n_m}\to y$, then we obtain that for each $g\in G$
$$d(\varphi_g(x), \varphi_g(y))<c.$$
Using expansivity of $\varphi$ it follows that $x=y$, But this is a contradiction to $d(x, y)\geq \epsilon$.
\end{proof}

\bigskip
\noindent Note that Lemma 2.10 in \cite{Lee} is a special case of above Lemma \ref{expc} as the latter is true even for non--compact space. Let $A$ be a subset of $X$. We say that an action $\varphi:G\times X \longrightarrow X$ is expansive on a subset $A$, if there is $c>0$ such that for every $x,y \in A$ with $x\neq y$ there is $g \in G$ satisfying $d(\varphi_g(x), \varphi_g(y))>c$. Also we say that non-empty subset $A\subseteq X$ satisfy Property \textrm{P},  if  for every $\epsilon>0$ there exists a compact subset $C\subseteq A$ of $X$ such that $\left(A\times A \cap d^{-1}([0, \epsilon))\right)\cup C\times C= A\times A$.
\begin{theorem}
  Suppose $X$ is a compact metric space and let $A$ be a dense subset of $X$ with Property \textrm{P} and let $\varphi:G\times X \longrightarrow X$ be an action. Then $\varphi$ is expansive on $A$ if and only  if
 $\varphi$ is expansive on $X$.
 \end{theorem}
 \begin{proof}
 It is sufficient to show that if $\varphi $ is expansive on $A$, then it is expansive on $X$. Let $c>0$ be expansive constant for $\varphi$ on $A$. Let $x, y\in X$ with $x \neq y$ and $\delta$ be such that $0 < \delta < \frac{c}{3}$. Then, we show that there is $g_0\in G$ such that $d(\varphi_{g_0}(x), \varphi_{g_0}(y)) > \delta$.   If $d(x, y)\geq \frac{c}{3}$, then we are through as in this case we take $g_0=e$.

  \medskip
  \noindent Let $x, y \in X$ be such that $d(x, y)< \epsilon$ where $0<\epsilon<\frac{c}{3}$.  Since $\overline{A}=X$, there is $\beta>0$ such that if $a\in B(x, \beta)\cap A$ and $b\in B(y, \beta)\cap A$ then $d(a, b)>\epsilon$. Since $A$ has  Property \textrm{P}, there is a finite set $F\subseteq G$ such that for  $a,b\in A$
 \begin{align*}\label{finite}
 d(a, b)>\epsilon \Longrightarrow d(\varphi_{g_0}(a), \varphi_{g_0}(b))>c
  \end{align*}
for  some  $g_0\in F$.   Further, $F$ is finite set and $\varphi$ is a continuous action, therefore for $\epsilon>0$ there is a  $\lambda >0$  such that for all $g\in F$
 \begin{align*}
 d(z, w)< \lambda \Longrightarrow d(\varphi_g(z), \varphi_g(w))<\epsilon
 \end{align*}
 Take $\mu= \min \{\beta, \lambda\}$ and choose $a\in B(x, \mu)\cap  A$ and $b\in B(y, \mu)\cap A$. Then  for all $g\in F$, $d(a, b)>\epsilon$,
 $$d(\varphi_g(x), \varphi_g(a))<\epsilon$$
 and
 $$d(\varphi_g(y),  \varphi_g(b))<\epsilon.$$
Moreover  there is $g_0\in F$ such that
 $$d(\varphi_{g_0}(a), \varphi_{g_0}(b))> c.$$
 Note that for this $g_0 \in G$
 \begin{align*}
 d(\varphi_{g_0}(x), \varphi_{g_0}(y)) \geq d(\varphi_{g_0}(a),  \varphi_{g_0}(b))- d(\varphi_{g_0}(a), \varphi_{g_0}(x)) -  d(\varphi_{g_0}(b), \varphi_{g_0}(y))  >c- 2\epsilon >\frac{c}{3}.
 \end{align*}
 Hence $\delta$ is an expansive constant for  $\varphi$ on $X$.
 \end{proof}
%
%

\medskip
\noindent If $c$ is an expansive constant for an action $\varphi$ of $G$ on a compact metric space $X$, then any $\delta$, $0< \delta < c$, is also an expansive constant for $\varphi$. If $\mathcal{F}$ is the set of all expansive constants for $\varphi$, then $\mathcal{F}$ is a bounded set of positive real numbers and hence least upper bound of $\mathcal{F}$ exists. In the following Theorem we show that the least upper bound of $\mathcal{F}$ is not an expansive constant.

\medskip
\begin{theorem}\label{const}
Let $X$ be a compact metric space and let $G$ be a finitely generated group with generating set $S$. Suppose  $\varphi:G\times X \longrightarrow X$ is expansive action. If $\beta$ is the least upper bound of the set $\mathcal{F}$ of all expansive constants for $\varphi$ then $\beta$ is not an expansive constant for $\varphi$.
\end{theorem}
\begin{proof}
For every $m\in \mathbb{N}$, $\beta + \frac{1}{m}$ is not an expansive constant $\varphi$ and therefore there is $x_m'\neq y_m'$ such that for all $g\in G$
$$d(\varphi_g(x_m'), \varphi_g(y_m'))<\beta + \frac{1}{m}.$$
Let $\delta>0$ be an expansive constant for $\varphi$. For every $m$, using expansivity of $\varphi$ there is $g_m\in G$ such that
$$d(\varphi_{g_m}(x_m'), \varphi_{g_m}(y_m'))>\delta.$$
Set $x_m=\varphi_{g_m}(x_m')$ and $y_m=\varphi_{g_m}(y_m')$. We assume that the sequence $\{x_m\}$ and  $\{y_m\}$ converge to  say, $x$ and $y$ respectively. Note that $x\neq y$.
Further, for  $g\in G$ and $\alpha>0$ choose $p,\: q\in \mathbb{N}$ and $\eta>0$ such that $\frac{1}{p}<\frac{\alpha}{3}$, $d(x, x_{g_q})<\eta$, $d(y, y_{g_q})<\eta$ and satisfying
\begin{equation}
d(a,b)<\eta \Rightarrow d(\varphi_g(a), \varphi_g(b))< \frac{\alpha}{3}.
\end{equation}
Therefore
$$
d(\varphi_g(x), \varphi_g(y)) \leq d(\varphi_g(x), \varphi_g(y_{q}'))+ d(\varphi_g(y_{q}'),\varphi_g(x_{q}'))+d(\varphi_g(x_{q}'), \varphi_g(y))   \leq  \alpha +\beta
$$
Above inequality, implies that there exist $x\neq y$ in $X$ such that for
any $g\in G$,   $d(\varphi_g(x), \varphi_g(y))\leq\beta$ and therefore $\beta$ is not expansive constant.
\end{proof}

\medskip
\noindent Note that Theorem \ref{const} is a generalization of Theorem 6 proved in \cite{bfb}. We now recall the definition of covering map.
\medskip
\begin{definition}
Let $X$, $Y$ be metric spaces. A continuous surjective map $f :X \longrightarrow Y$ is called a \emph{covering map} if there is an open cover $\{U_{\alpha}\}$ of $Y$  such that for every $\alpha$, $f^{-1}(U_{\alpha})$ is a disjoint union of open sets in $X$ and each of this open subset of $X$ is mapped homeomorphically onto $U_{\alpha}$ by $f$. It $X=Y$, then we say $f$ is a self--covering map.
\end{definition}

\medskip
\noindent Recall, a continuous surjection $f:X \longrightarrow Y$ is called a \emph{local homeomorphism} if for each $x \in X$ there is an open neighbourhood $U_x$ of $x$ such that $f(U_x)$ is open in $Y$ and $f_{| U_x} ; U_x \longrightarrow f(U_x)$ is a homeomorphism. Obviously, every covering map is local homeomorphism, but in general converse is not true. Note that converse is true if the space $X$ is a compact metric space. For the completion we give the proof of the following Lemma, which is required to prove the next Theorem.

\bigskip
\begin{lemma}\label{cover}
Let $(X,d)$ and $(Y, \rho)$ be two metric spaces. Suppose $f:X\longrightarrow Y$ is a covering map. If $X$ is compact,  then there is a real number  $\beta$, $\beta>0$, such that $d(x,y)>\beta$, whenever $f(x)=f(y)$ and $x\neq y$.
\end{lemma}
\begin{proof}
Given that $f$ is a  covering map and therefore $f$ is a local homeomorphism. Hence,  for each $z\in X$, there is an open set $U_z$ of $z$ such that $f(U_z)$ is open in $Y$ and $f_{|U_z} : U_z \longrightarrow f(U_z)$ is a homeomorphism. If possible, suppose  there are sequences $\{x_n\}$ and $\{y_n\}$ in $X$ such that $f(x_n)=f(y_n)$, $x_n \neq y_n$ and $d(x_n, y_n)<\frac{1}{n}$. We denote the convergent subsequences again by  $\{x_n\}$, and $\{y_n\}$. If $\{x_n\}$ converges to $x$ and $\{y_n\}$ converges to $y$, then using local homeomorphism of $f$ and $f(x_n)=f(y_n)$, we obtain $x=y$. Thus for every neighbourhood $U_x$ of $x$ (and hence of $y$), there exists $x_n$, $y_n$ in $U_x$ with $x_n \neq y_n$ and $f(x_n)=f(y_n)$. But this implies that there is no neighbourhood $U_x$ of $x$ such that $f_{|U_x}$ is one--one and hence for every neighbourhood $U_x$ of $x$ $f_{|U_x} : U_x \longrightarrow f(U_x)$ is not a homeomorphism. This is a contradiction.
\end{proof}

\bigskip
\noindent In the following theorem we obtain condition for expansivity of an action using the covering map.
\medskip
\begin{theorem}\label{covering}
Let $(X,d)$ and $(Y, \rho)$ be compact metric spaces. Suppose that \linebreak $\varphi : G\times X \longrightarrow X$ is an uniformly continuous action of $G$ on $X$ and \linebreak $\psi : G\times Y\longrightarrow Y$ is an expansive action of $G$ on $Y$. Let  $f:X\longrightarrow Y$ be a covering map.  If for each $g \in G$, $f \circ \varphi_g=\psi_g\circ f$, then  $\varphi$ is also an  expansive action.
\end{theorem}
\begin{proof}
For $\delta>0$ and $x\in X$, denote
$$ \Gamma_{\delta}(x, \varphi)= \{t \in X: d(\varphi_g(x), \varphi_g(t))<\delta, \forall g\in G\}$$
Let $c$ be an expansivity constant for $\psi$. Then for each $y \in Y$, $\Gamma_c(y, \psi)=\{y\}$. If possible suppose $\varphi$ is not  expansive. Then for each $\epsilon>0$ there correspond an $x_{\epsilon}\in X$, such that $\Gamma_{\epsilon}(x_{\epsilon},\varphi)\neq \{x_{\epsilon}\}$. Further, since $f$ is a covering map, there is a  $\beta>0$ satisfying the Lemma \ref{cover}. Let $\epsilon$ be such that $0<\epsilon<\frac{\beta}{2}$. Then for every $x,y\in \Gamma_{\epsilon}(x_{\epsilon},\varphi)$ we have $d(x,y)< \beta$. Therefore by Lemma \ref{cover} it follows that $f(x)\neq f(y)$.

\bigskip
\noindent  Choose a finite open covering $\{U_i\}$ of $X$ such that $diam (f(U_i))<\frac{c}{2}$ for all $i=1,\ldots, n$. Note that such a finite open cover exists as both $X$,  $Y$ are compact and $f$ is a continuous surjective map. Further, let $\epsilon_0>$ be a Lebesgue number of the covering $\{U_i\}$. Let $\epsilon$ be such that $0 < \epsilon < min\left\{\epsilon_0, \frac{\beta}{2}\right\}$.  Then for every $g\in G$ and every $y_{\epsilon}\in \Gamma_{\epsilon}(x_{\epsilon},\varphi)$, it follows that
   $$\varphi_g(y_{\epsilon}), \varphi_g(x_{\epsilon})\in U_i, \;\; \mbox{ for some } i$$
 But $diam (f(U_i))< \frac{c}{2}$. Therefore,
   $$\rho\left(f(\varphi_g(y_{\epsilon})), f(\varphi_g(x_{\epsilon}))\right)<\frac{c}{2}$$
Using $f\circ \varphi_g=\psi_g\circ f$, we obtain that for every $g \in G$
\begin{equation}\label{2}
d\left(\psi_g(f(x_{\epsilon})), \psi_g(f(y_{\epsilon}))\right)<\frac{c}{2}< c
\end{equation}
But $\psi$ is expansive with expansive constant $c$. Therefore above inequality implies $f(x_{\epsilon})=f(y_{\epsilon})$. Also, $y_{\epsilon}\in  \Gamma_{\epsilon}(x_{\epsilon},\varphi)$ and $\epsilon < \frac{\beta}{2}$. Therefore, by Lemma \ref{cover}, it follows that $f(x_{\epsilon})\neq f(y_{\epsilon})$, which is a contradiction. Hence the proof.
\end{proof}

\bigskip
\noindent Proof of the following properties follows easily, so we only state them. Recall, a point $x \in X$ is a fixed point of an action $\varphi : G\times X\longrightarrow X$ if for each $g \in G$, $\varphi_g(x)=x$. Let $Fix \:\varphi$ denote the set of all fixed points of $\varphi$.
\begin{proposition}
Let $(X,d)$ be a compact metric spaces and let $G$ be a group. Suppose $\varphi : G\times X\longrightarrow X$   is an expansive action. Then following hold:
\begin{itemize}
  \item
  $Fix \:\varphi$ is a finite set.
  \item Let $Y$ be a closed $\varphi -$invariant (i.e. $\varphi_g(Y)\subseteq Y$, for all $g \in G$) subset of $X$. Then $\varphi : G\times Y \longrightarrow Y$ is also expansive.
\end{itemize}

\end{proposition}

\section{Orbit Expansive Group Action}
\medskip
\noindent Let $(X, \tau)$ be a topological space. For a subset $A\subseteq X$ and a cover $\mathcal{U}$ of $X$ we write $A\prec \mathcal{U}$ if there exists $C\in \mathcal{U}$ such that $A\subseteq C$. If $\mathcal{V}$ is a family of subsets of $X$, then $\mathcal{V}\prec \mathcal{U}$ means that $A\prec \mathcal{U}$ for all $A\in\mathcal{V}$. If, in addition $\mathcal{V}$ is a cover of $X$, then $\mathcal{V}$ is said to be \emph{refinement of} $\mathcal{U}$. Join of two covers $\mathcal{U}$ and $\mathcal{V}$ is a cover given by $\mathcal{U}\vee \mathcal{V}=\{U\cap V| U\in\mathcal{U}, V\in\mathcal{V}\}$. The notion for orbit expansivity for $\Z-$action was first defined in \cite{art}. In the following we define the concept for action of any group $G$.

 \medskip
 \begin{definition}
 Let $\varphi:G\times X \longrightarrow X$ be a continuous action of group $G$ on a topological space $X$. Then $\varphi$ is said to be \emph{an orbit  expansive action} if there is a finite open cover $\mathcal{U}$ of $X$ such that if  for each $g \in G$, the set $\{\varphi_g(x), \varphi_g(y)\}\prec \mathcal{U}$, then $x=y$. The cover $\mathcal{U}$ of $X$  is called \emph{an orbit  expansive covering of} $\varphi$.
 \end{definition}

\medskip
\noindent Equivalently, $\varphi$ is an orbit expansive action if whenever $x, y \in X$ with $x\neq y$ there exists $g \in G$ such that $\{\varphi_g(x), \varphi_g(y)\} \not\prec \mathcal{U}$.  It is clear that if $\mathcal{U}$ is an orbit  expansive covering of $\varphi$ and $\mathcal{V}$ is a refinement of $\mathcal{U}$, then $\mathcal{V}$ is an orbit  expansive covering of $\varphi$.  Since $\mathcal{U}\vee \mathcal{V}$ is a refinement of $\mathcal{U}$, it follows that if $\mathcal{U}$ is an orbit  expansive covering then so is $\mathcal{U}\vee \mathcal{V}$,  for every open cover $\mathcal{V}$ of $X$. Recall, for each $g \in G$, $\varphi_g : X \longrightarrow X$ is a homeomorphism. Therefore, $\varphi_g(\mathcal{U})$ is an open cover of $X$. It is easy to verify that if $\mathcal{U}$ is an orbit expansive covering of $\varphi$, then so is $\varphi_g(\mathcal{U})$. Further,  $\vee_{g\in G} \; \varphi_g(\mathcal{U})$ is also an orbit  expansive covering of $\varphi$. In the following theorem we characterize expansive group action using orbit expansive group action.

\medskip
 \begin{theorem}\label{equ}
 Let $(X, d)$ be a compact metric space and let $\varphi:G\times X \longrightarrow X$ be an uniformly continuous action of $G$ on $X$. Then $\varphi$ is orbit  expansive if and only if it is   expansive.
 \end{theorem}

\begin{proof}
    Suppose $\varphi$ is expansive with expansive constant $c$.  Since $X$ is a compact metric space, it follows that  there are $x_1, x_2, \dots, x_n\in X$ such that $\mathcal{V}=\{B(x_i, \frac{c}{2}): i=1, \ldots, n\}$ is a finite open cover for $X$. We show that $\mathcal{V}$ is an orbit expansive covering of $\varphi$.
    For $c>0$ and $x\in X$, denote
$$ \Gamma_c(x, \varphi)= \{t \in X: d(\varphi_g(x), \varphi_g(t))<c, \forall g\in G\}$$
    Then expansivity of $\varphi$ implies that for each $x \in X$, $\Gamma_c(x, \varphi)=\{x\}$.
\medskip
\noindent    For every $g \in G$, let $\{\varphi_g(x), \varphi_g(y)\}\prec \mathcal{V} $.  Then there is $B \in \mathcal{V}$, such that $\{\varphi_g(x), \varphi_g(y)\} \subseteq B$. But $diam \; B < c$. Therefore for each $g \in G$,
$$d(\varphi_g(x), \varphi_g(y)) < c.$$
This implies $y \in \Gamma_c(x, \varphi)$. But $\Gamma_c(x, \varphi)=\{x\}$. Therefore $x=y$. Hence $\varphi$ is orbit expansive.

\bigskip
\noindent Conversely, suppose $\varphi$ is orbit expansive with orbit expansive covering $\mathcal{V}$. Since $X$ it compact, without loss of generality we can assume that $\mathcal{V}$ is a finite open covering of $X$. Let $\delta>0$ be Lebesgue number for open cover $\mathcal{V}$. Then we show that $\varphi$  is expansive with expansive constant $c$, where $c$ is such that $0<c<\frac{\delta}{2}$.  Let every $x, y\in X$ be such that for each $g \in G$,
$$d(\varphi_g(x), \varphi_g(y)) < c. $$
Then since $\delta$ is a Lebesgue number for $\mathcal{V}$, it follows that there is $U \in \mathcal{V}$, such that $\{\varphi_g(x), \varphi_g(y)\} \subseteq U$. Therefore, for each $g \in G$,
$$\{\varphi_g(x), \varphi_g(y)\}\prec \mathcal{V}.$$
But $\mathcal{V}$ is orbit expansive covering of $\varphi$. Therefore $x=y$. Hence $\varphi$ is expansive with expansive constant $c$.
 \end{proof}

\medskip
\noindent We now give an example of an orbit expansive group action on a non--Hausdorff $T_1-$space. We use technique in \cite[Example 3.13]{art} and extend it. Recall, a continuum is a connected compact non--empty metric space. Further, a continuum $X$ is said to be \emph{1--arcwise connected continuum} if for any $x, y \in X$ with $x\neq y$, there is a unique arc in $X$ with end points $x$ and $y$. We recall the following Theorem from \cite{shi}.

\medskip
\begin{lemma}\label{fix}
  Let $X$ be a 1--arcwise connected continuum and let $G$ be a nilpotent group. Then the action $\varphi:G\times X\longrightarrow X$ has a fixed point in $X$.
\end{lemma}

\smallskip
\begin{example}
Let $X$ be a 1--arcwise connected continuum with no--isolated point and  metric $d$. Suppose the metric $d$ induce topology $\tau$ on $X$ and suppose $G$ is a nilpotent group. Let $\varphi:G\times X \longrightarrow X$ be an expansive group action with a fixed point $x_0\in X$. Note that existence of fixed point is guaranteed by Lemma \ref{fix}. Consider a new point $x_1$ such that $x_1\notin X$ and set $\overline{X}=X\cup \{x_1\}$. Define topology $\overline{\tau}$ on $\overline{X}$ as follows:
A subset $W$ of $\overline{X}$ is in $\overline{\tau}$ if
\begin{enumerate}
  \item $W \in \tau$ or
  \item $W= U\cup\{x_1\}$, where $ x_0\in U$ and $U\in \tau$ or
  \item $ W= \left(U\setminus\{x_0\}\right)\cup \{x_1\}$, where $x_0\in U$ and $U\in\tau$
\end{enumerate}
Then the space $(\overline{X}, \overline{\tau})$ is a non-Hausdorff $T_1-$space \cite{art}. Let $\psi:G\times \overline{X}\longrightarrow \overline{X}$ be an action of $G$ on $\overline{X}$  given by $\psi(g, x)= \varphi(g, x)$ if $x\in X$ and $\psi(g, x_1)=x_1$ for all $g\in G$. Since $\varphi$ is expansive group action on compact metric space $X$, it follows  by Theorem \ref{equ} that $\varphi$ is an orbit expansive action. Therefore  there is an orbit expansive covering $\mathcal{U}=\{U_1, U_2, \ldots, U_n\}$ of $\varphi$. Suppose $x_0\in U_n$. Set $U_{n+1}= \left(U_n-\{x_0\}\right)\cup \{x_1\}$. Then $U_{n+1}$ is open in $\overline{X}$.  We show that $\mathcal{V}=\{U_1, U_2, \ldots, U_n, U_{n+1}\}$ is an orbit expansive covering of $\psi$. For $x, y \in \overline{X}$, let $\{\psi_g(x), \psi_g(y)\}\prec \mathcal{V}$, for all $g\in G$. If $x, y \in X$, then using $\mathcal{U}$ is an orbit expansive covering of $\varphi$  implies that $x=y$.  Next, let $x=x_1$ and $y \in X$.   Then $\{\psi_g(y), x_1\}\subseteq U_{n+1}$. Therefore $y\neq x_0$. Further, $y \in X$ and also for all $g \in G$ $\psi_g(y) = \varphi_g(y)\in U_n$ will imply that $\{\varphi_g(y), \varphi_g(x_0)\}\subseteq U_n$ for all $g\in G$. But $\mathcal{U}$ is orbit expansive covering of $\varphi$.  Therefore $y=x_0$ which is a contradiction. Hence in any case $\{\psi_g(x), \psi_g(y)\}\prec \mathcal{V}$ for all $g\in G$ implies that $x=y$.
\end{example}

\medskip
\noindent In the following we prove that if there exists an orbit expansive group action on a topological space, then the space is always a $\mathrm{T}_1-$space. Note that this result to similar to the Proposition 2.5 of \cite{art}.

\medskip
 \begin{theorem}
  Let $\varphi:G\times X \longrightarrow X$ be a continuous action of group $G$ on a topological space $X$. If $\varphi$ is an orbit expansive group action, then $X$ is always $\mathrm{T}_1-$space.
 \end{theorem}
 \begin{proof}
 Let $x, y \in X$ with $x \neq y$. Since $\varphi$ is an orbit expansive action, there is an orbit expansive covering  $\mathcal{U}$ of $X$ such that for some $g \in G$
 $$\{\varphi_g(x), \varphi_g(y)\} \not \prec \mathcal{U}.$$
  But this implies that if for some $U\in \mathcal{U}$,  $\varphi_g(x)\in U$, then $x\in \varphi_{g^{-1}}(U)$ and $y\notin \varphi_{g^{-1}}(U)$. Also, $\varphi_{g^{-1}} : X \longrightarrow X$ is a homeomorphism. Therefore there is open set $\varphi_{g^{-1}}(U)$ of $X$ such that $x\in \varphi_{g^{-1}}(U)$ but $y\notin \varphi_{g^{-1}}(U)$. Hence $X$ is a $\mathrm{T}_1-$space.
 \end{proof}

\smallskip
 \noindent If $\mathcal{U}=\{U_1, U_2, \ldots, U_n\}$ is an orbit expansive covering of $\varphi:G\times X\longrightarrow X$, then for $Y\subseteq X$ with $G(Y)=Y$, the open cover  $\{U_1\cap Y, U_2\cap Y, \ldots, U_n\cap Y\}$ is an orbit  expansive covering of $\varphi_{|Y} :G\times Y\longrightarrow Y$. Thus if $\varphi$ is an orbit expansive action of $G$ on $X$ and $Y$ is a $G-$invariant subset of $X$, then $\varphi_{|Y}$ is also an orbit expansive action of $G$ on $Y$. Next, if $H$ is a subgroup of group $G$ and $\varphi_{|H}:H\times X\longrightarrow X$ is an orbit expansive action then so is $\varphi:G\times X\longrightarrow X$. In the following we show that converse is true if $H$ is a finite index subgroup of $G$.

\smallskip
 \begin{theorem}
 Let $(X, \tau)$ be a topological space and $\varphi:G\times X\longrightarrow X$ be an action of $G$ on $X$. Suppose $H$ is finite index subgroup of $G$. Then  $\varphi_{|H}:H\times X\longrightarrow X$ is orbit expansive action if $\varphi:G\times X \longrightarrow X$ is orbit expansive.
 \end{theorem}
 \begin{proof}
Suppose that $\mathcal{U}=\{U_1, U_2, \ldots, U_n\}$ is an orbit expansive covering of $\varphi$. Since $H$ is finite index subgroup of $G$, there is $\{g_1, g_2, \ldots, g_k\}\subseteq G$ such that $G= \bigcup_{i=1}^{k}g_iH$. We show that $\mathcal{V}=\left\{V_m^i = g_i^{-1}U_m : m\in\{1,\ldots, n\}\;  \& \; i \in \{1, 2, \ldots, k\}\right\}$ is an orbit expansive covering of $\varphi_{|H}$. Orbit expansivity of $\varphi$ implies that there is $g \in G$ such that $$\{\varphi_g(x), \varphi_g(y)\} \not\prec \mathcal{U}.$$
For this $g \in G$, there is $i \in \{1, 2, \ldots, k\}$ and $h \in H$ such that $g=g_ih$. It is now easy to show that for this $h \in H$,
$$\{\varphi_h(x), \varphi_h(y)\} \not \prec \mathcal{V}.$$
Hence, $\varphi_{|H}$ is an orbit expansive action.
 \end{proof}

\medskip
\noindent Let $(X, d)$ and $(Y, \rho)$ be two topological spaces. Suppose $\varphi:G\times X \longrightarrow X$ and $\psi:G\times Y \longrightarrow Y$ are two uniformly continuous conjugate actions with conjugancy $h$. Therefore $h : X \longrightarrow Y$ is a homeomorphism with $ho\varphi_g=\psi_goh$ for all $g \in G$. If $\mathcal{U}$ is a finite open cover of $X$ such that $h(\mathcal{U})$ is an orbit  expansive covering of $\psi$, then $\mathcal{U}$ is an orbit  expansive covering of $\varphi$. We therefore have the following theorem:

\smallskip
 \begin{theorem}
 Let $(X, d)$ and $(Y, \rho)$ be two topological spaces and let $G$ be a group. Suppose $\varphi:G\times X \longrightarrow X$ and $\psi:G\times Y \longrightarrow Y$ are two uniformly continuous conjugate actions with conjugancy $h$. Then $\varphi$ is orbit  expansive if and only if so is $\psi$.
 \end{theorem}

\end{large}
\end{document}